\documentclass[11pt]{amsart}

\theoremstyle{plain}
\newtheorem{theorem}                {Theorem}      [section]
\newtheorem{proposition}  [theorem]  {Proposition}
\newtheorem{corollary}    [theorem]  {Corollary}
\newtheorem{lemma}        [theorem]  {Lemma}

\theoremstyle{definition}

\newtheorem{remark}       [theorem]  {Remark}
\newtheorem{definition}   [theorem]  {Definition}

 \DeclareMathOperator{\id}{I}

\DeclareMathOperator{\Span}{span}
 
\DeclareMathOperator{\cst}{constant}
 
 \DeclareMathOperator{\im}{Im}

\numberwithin{equation}{section}

\begin{document}

\title[Surfaces with parallel mean curvature vector]
{Surfaces with parallel mean curvature in $\mathbb{C}P^n\times\mathbb{R}$ and $\mathbb{C}H^n\times\mathbb{R}$}

\author{Dorel~Fetcu}
\author{Harold~Rosenberg}

\address{Department of Mathematics\\
"Gh. Asachi" Technical University of Iasi\\
Bd. Carol I no. 11 \\
700506 Iasi, Romania} \email{dfetcu@math.tuiasi.ro}

\curraddr{IMPA\\ Estrada Dona Castorina\\ 110, 22460-320 Rio de
Janeiro, Brasil} \email{dorel@impa.br}

\address{IMPA\\ Estrada Dona Castorina\\ 110, 22460-320 Rio de
Janeiro, Brasil} \email{rosen@impa.ro}


\begin{abstract} We consider surfaces with parallel mean curvature vector (pmc surfaces) in $\mathbb{C}P^n\times\mathbb{R}$ and $\mathbb{C}H^n\times\mathbb{R}$, and, more generally, in
cosymplectic space forms. We introduce a holomorphic quadratic differential on such surfaces.
This is then used in order to show that the anti-invariant pmc $2$-spheres of a $5$-dimensional non-flat
cosymplectic space form of product type are actually the embedded rotational spheres $S_H^2\subset\bar M^2\times\mathbb{R}$ of Hsiang and Pedrosa, where $\bar M^2$ is a complete
simply-connected surface with constant curvature. When the ambient space is a cosymplectic space form of product type and its dimension is greater than $5$, we prove that an immersed non-minimal non-pseudo-umbilical anti-invariant $2$-sphere lies in a product space $\bar M^4\times\mathbb{R}$, where $\bar M^4$ is a space form. We also provide a reduction of codimension theorem for the pmc surfaces of a non-flat cosymplectic space form.
\end{abstract}

\date{}



\maketitle

\section{Introduction}

Surfaces with constant mean curvature (cmc surfaces) in $3$-dimensional ambient spaces have been intensively studied in the last six decades and a very useful tool proved to be the holomorphic quadratic forms defined on such surfaces.

In 1951, H. Hopf used for the first time a holomorphic quadratic form in order to show that any cmc surface in a Euclidean space, homeomorphic to a sphere, is actually a round sphere (see \cite{HH}) and then his result was extended to cmc surfaces in $3$-dimensional space forms by S.-S.~Chern, in \cite{C}. When the codimension is greater than $1$, a natural generalization of cmc surfaces are surfaces with parallel mean curvature vector (pmc surfaces). These surfaces are studied since the early seventies, among the first papers to treat this subject being \cite{DF} by D.~Ferus, \cite{CL} by B.-Y.~Chen and G.~D.~Ludden, \cite{DH} by D.~A.~Hoffman and \cite{Y} by S.-T.~Yau. In this last paper it is proved that a pmc surface immersed in a space form either lies in a totally geodesic $3$-dimensional space or it is a minimal surface of an umbilical hypersurface.

The next natural step was taken by U. Abresch and H. Rosenberg, who studied in \cite{AR,AR2} cmc surfaces and obtained Hopf type results in product spaces of type $M^2(\rho)\times\mathbb{R}$, where $M^2(\rho)$ is a complete simply-connected surface with constant curvature $\rho$, as well as the homogeneous $3$-manifolds $Nil(3)$, $\widetilde{PSL(2,\mathbb{R})}$ and Berger spheres. As for the study of pmc surfaces in product spaces of type $M^n(\rho)\times\mathbb{R}$, where $M^n(\rho)$ is a space form with constant sectional curvature $\rho$, the papers \cite{AdCT1} and \cite{AdCT} by H. Alencar, M. do Carmo and R. Tribuzy, are devoted to this subject. The principal tool they use is a holomorphic quadratic form, which in the $3$-dimensional case is just the Abresch-Rosenberg differential, introduced in \cite{AR}. In \cite{AdCT} the authors proved, amongst others, a very nice reduction of the codimension theorem, showing that a pmc surface immersed in $M^n(\rho)\times\mathbb{R}$ is either a minimal surface in a totally umbilical hypersurface of $M^n(\rho)$; a cmc surface in a $3$-dimensional totally umbilical submanifold, or in a totally geodesic submanifold of $M^n(\rho)$; or it lies in $M^4(\rho)\times\mathbb{R}$. In the recent paper \cite{F} a similar result is proved for pmc surfaces immersed in a complex space form, i.e. a K\"ahler manifold with constant holomorphic sectional curvature. There it is shown that a non-minimal pmc surface immersed in a non-flat complex space form $N^n(\rho)$, where $\rho$ is the (constant) holomorphic sectional curvature and $n\geq 3$, is a pseudo-umbilical totally real surface or it lies in a complex space form $N^r(\rho)$, with $r\leq 5$.

The products between a complex space form and a one dimensional manifold are the main examples of cosymplectic space forms, which are often seen as the odd-dimensional version of complex space forms. Therefore, working in such spaces seems to be the natural continuation of \cite{F}. The other option for odd-dimensional ambient spaces with nice curvature properties is represented by the Sasakian space forms, amongst them being the odd-dimensional spheres and the generalized Heisenberg group. Although the present paper is devoted to the study of pmc surfaces in cosymplectic space forms it is sure that interesting results could be also obtained by considering this second option.

The paper is organized as follows. In Section \ref{sintro} we briefly recall
some general facts about the cosymplectic space forms, as they are presented in \cite{ABC,B,BG,CdLM}. In Section \ref{sqf} we introduce a quadratic form $Q$ defined on surfaces immersed in such a space and prove that its
$(2,0)$-part is holomorphic when the mean curvature vector of the
surface is parallel. In Section \ref{scyl} we characterize the pmc surfaces of type $\Sigma^2=\pi^{-1}(\gamma)$ in a product space $M^n(\rho)\times\mathbb{R}$, where $M^n(\rho)$ is a complex space form, $\pi:M^n(\rho)\times\mathbb{R}\rightarrow M^n(\rho)$ is the projection map and $\gamma:I\rightarrow M^n(\rho)$ is a Frenet curve of osculating order $r$ in $M^n(\rho)$. We also prove that such surfaces with vanishing $(2,0)$-part of $Q$ exist if and only if $\rho<0$.
The main result of Section \ref{sred} is a reduction theorem, which
states that a non-minimal pmc surface $\Sigma^2$ in a non-flat cosymplectic space form $N^{2n+1}(\rho)$ either is pseudo-umbilical and then the characteristic vector field is orthogonal to $\Sigma^2$ and the surface is anti-invariant, or it
is not pseudo-umbilical and lies in a totally geodesic invariant submanifold of $N^{2n+1}(\rho)$ with dimension less than or equal to $11$. The last Section is devoted to the study of anti-invariant pmc surfaces. We prove that any non-minimal anti-invariant pmc $2$-sphere in $M^2(\rho)\times\mathbb{R}$ is an embedded rotationally invariant cmc sphere $S_H^2\subset\bar M^2(\frac{\rho}{4})\times\mathbb{R}$, where $\bar M^2(\frac{\rho}{4})$ is a complete simply-connected surface with constant curvature $\frac{\rho}{4}$, immersed as a totally-geodesic Lagrangian submanifold in the complex space form $M^2(\rho)$. When the dimension of the ambient space is greater than $5$, we show that a non-minimal non-pseudo-umbilical anti-invariant $2$-sphere immersed in $M^n(\rho)\times\mathbb{R}$ lies in a product space $\bar M^4(\frac{\rho}{4})\times\mathbb{R}$, where $\bar M^4(\frac{\rho}{4})$ is a space form immersed as a totally geodesic totally real submanifold in $M^n(\rho)$.

\section{Preliminaries}\label{sintro}

Let $M^n(\rho)$ be a complex space form with the complex structure $(J,\langle,\rangle_M)$, consider the product manifold $N^{2n+1}=M^n(\rho)\times\mathbb{R}$ and define the following tensors on $N^{2n+1}$:
$$
\varphi=J\circ d\pi,\quad\xi=\frac{\partial}{\partial t},\quad\eta=dt\quad\textnormal{and}\quad
\langle,\rangle_N=\langle,\rangle_M+dt\otimes dt,
$$
where $\pi:M^n(\rho)\times\mathbb{R}\rightarrow M^n(\rho)$ is the projection map and $t$ is the standard coordinate function on the real axis. Then $(N^{2n+1},\varphi,\xi,\eta,\langle,\rangle_N)$ is a cosymplectic space form with constant $\varphi$-sectional curvature equal to $\rho$ (see \cite{ABC,BG}). We shall explain what this means in the following.

An \textit{almost contact metric structure} on an odd-dimensional manifold
$N^{2n+1}$ is given by $(\varphi,\xi,\eta,\langle,\rangle)$, where $\varphi$ is
a tensor field of type $(1,1)$ on $N$, $\xi$ is a vector field,
$\eta$ is its dual $1$-form and $\langle,\rangle$ is a Riemannian metric such that
$$
\varphi^{2}U=-U+\langle U,\xi\rangle\xi\quad\textnormal{and}\quad
\langle\varphi U,\varphi V\rangle=\langle U,V\rangle-\eta(U)\eta(V),
$$
for all tangent vector fields $U$ and $V$.

\noindent An almost contact metric structure $(\varphi,\xi,\eta,\langle,\rangle)$ is
called {\it normal} if
$$
N_{\varphi}(U,V)+2d\eta(U,V)\xi=0,
$$
where
$$
N_{\varphi}(U,V)=[\varphi U,\varphi V]-\varphi \lbrack \varphi
U,V]-\varphi \lbrack U,\varphi V]+\varphi^{2}[U,V],
$$
is the Nijenhuis tensor field of $\varphi$.

An almost contact metric manifold $(N,\varphi,\xi,\eta,g)$ is a
\textit{cosymplectic manifold} if it is normal and both the $1$-form $\eta$ and the fundamental $2$-form $\Omega$, defined by $\Omega(U,V)=\langle U,\varphi V\rangle$, are closed. Equivalently, an almost contact metric manifold is cosymplectic if and only if $\varphi$ is parallel, i.e. $\nabla^N\varphi=0$, where $\nabla^N$ is the Levi-Civita connection. This implies that also the vector field $\xi$ and the $1$-form $\eta$ are parallel. We note that a cosymplectic manifold has a natural local product structure as a product of a K\"ahler manifold and a $1$-dimensional manifold but there exist compact cosymplectic manifolds which are not global products (see \cite{B,CdLM}). We also recall that a submanifold $M$ of a cosymplectic manifold is called \textit{invariant} when $\varphi(TM)\subset TM$ and \textit{anti-invariant} when $\varphi(TM)\subset NM$, where $NM$ is the normal bundle of $M$.

Let $(N,\varphi,\xi,\eta,\langle,\rangle)$ be a cosymplectic manifold. The
sectional curvature of a $2$-plane generated by $U$ and $\varphi U$,
where $U$ is a unit vector orthogonal to $\xi$, is called
\textit{$\varphi$-sectional curvature} determined by $U$. A cosymplectic
manifold with constant $\varphi$-sectional curvature $\rho$ is called a
\textit{cosymplectic space form} and is denoted by $N(\rho)$. The curvature
tensor field of a cosymplectic space form $N(\rho)$ is given by
\begin{equation}\label{eq:curv}
\begin{array}{lcl}
R^N(U,V)W&=&\frac{\rho}{4}\{\langle V,W\rangle U-\langle U,W\rangle V+\langle U,\varphi W\rangle \varphi V-\langle V,\varphi W\rangle\varphi U\\
\\&&+2\langle U,\varphi V\rangle\varphi W+\eta(U)\eta(W)V-\eta(V)\eta(W)U\\ \\&&+\langle U,W\rangle\eta(V)\xi-\langle V,W\rangle\eta(U)\xi\}.
\end{array}
\end{equation}

\section{A quadratic form with holomorphic $(2,0)$-part}\label{sqf}

Although our main interest is to study the immersed pmc surfaces in product spaces of type $M^n(\rho)\times\mathbb{R}$, where $M^n(\rho)$ is a complex space form, it is more convenient to treat the more general case where the surfaces are immersed in an arbitrary cosymplectic space forms.

Let $\Sigma^2$ be an immersed surface in a cosymplectic space form $N^{2n+1}(\rho)$,  endowed with the cosymplectic
structure $(\varphi,\xi,\eta,\langle,\rangle)$ and having constant $\varphi$-sectional curvature $\rho$.

\begin{definition} If the mean curvature vector $H$ of the surface $\Sigma^2$ is
parallel in the normal bundle, i.e.
$\nabla^{\perp}H=0$, the normal connection $\nabla^{\perp}$ being
defined by the equation of Weingarten
$$
\nabla^{N}_XV=-A_VX+\nabla^{\perp}_XV,
$$
for any vector field $X$ tangent to $\Sigma^2$ and any vector
field $V$ normal to the surface, where $\nabla^{N}$ is the
Levi-Civita connection on $N$ and $A$ is the shape operator, then $\Sigma^2$ is called a \textit{pmc surface}.
\end{definition}

We define a quadratic form $Q$ on
$\Sigma^2$ by
$$
Q(X,Y)=8|H|^2\langle\sigma(X,Y),H\rangle-\rho|H|^2\eta(X)\eta(Y)+3\rho\langle \varphi X,
H\rangle\langle \varphi Y, H\rangle,
$$
where $\sigma$ is the second fundamental form of the surface,
and claim that the $(2,0)$-part of $Q$ is holomorphic.

In order to prove this, we first consider the isothermal coordinates
$(u,v)$ on $\Sigma^2$. Then $ds^2=\lambda^2(du^2+dv^2)$ and let us define
$z=u+iv$, $\widehat z=u-iv$, $dz=\frac{1}{\sqrt{2}}(du+idv)$, $d\widehat
z=\frac{1}{\sqrt{2}}(du-idv)$ and
$$
Z=\frac{1}{\sqrt{2}}\Big(\frac{\partial}{\partial
u}-i\frac{\partial}{\partial v}\Big),\quad \widehat Z=\frac{1}{\sqrt{2}}\Big(\frac{\partial}{\partial
u}+i\frac{\partial}{\partial v}\Big).
$$
We get $\langle Z,\widehat Z\rangle=\langle\frac{\partial}{\partial
u},\frac{\partial}{\partial
u}\rangle=\langle\frac{\partial}{\partial
v},\frac{\partial}{\partial v}\rangle=\lambda^2$. We mention that this rather unusual notation for the conjugation is used only for the reader's convenience.

Now, we shall compute
$$
\widehat Z(Q(Z,Z))=\widehat Z(8|H|^2\langle\sigma(Z,Z),H\rangle-\rho|H|^2(\eta(Z))^2+3\rho\langle \varphi Z, H\rangle^2).
$$

We have
$$
\begin{array}{ll}
\widehat Z(\langle\sigma(Z,Z),H\rangle)&=\langle\nabla^N_{\widehat Z}\sigma(Z,Z),H\rangle+\langle\sigma(Z,Z),\nabla^N_{\widehat Z}H\rangle\\
\\&=\langle\nabla^{\perp}_{\widehat Z}\sigma(Z,Z),H\rangle+\langle\sigma(Z,Z),\nabla^{\perp}_{\widehat Z}H\rangle\\ \\&=\langle(\nabla^{\perp}_{\widehat Z}\sigma)(Z,Z),H\rangle+\langle\sigma(Z,Z),\nabla^{\perp}_{\widehat Z}H\rangle,
\end{array}
$$
since
$$
(\nabla^{\perp}_{\widehat Z}\sigma)(Z,Z)=\nabla^{\perp}_{\widehat Z}\sigma(Z,Z)-2\sigma(\nabla_{\widehat Z}Z,Z)=\nabla^{\perp}_{\widehat Z}\sigma(Z,Z)
$$
and $\nabla_{\widehat Z}Z=0$, from the definition of the connection $\nabla$ on the
surface.

Next, using the Codazzi equation, we get
\begin{equation}\label{eq:1}
\begin{array}{lll}
\widehat Z(\langle
\sigma(Z,Z),H\rangle)&=&\langle(\nabla^{\perp}_{Z}\sigma)(\widehat Z,Z),H\rangle+\langle (R^N(\widehat Z,Z)Z)^{\perp},H\rangle\\ \\&&+\langle\sigma(Z,Z),\nabla^{\perp}_{\widehat Z}H\rangle\\
\\&=&\langle(\nabla^{\perp}_{Z}\sigma)(\widehat Z,Z),H\rangle+\langle R^N(\widehat Z,Z)Z,H\rangle+\langle\sigma(Z,Z),\nabla^{\perp}_{\widehat Z}H\rangle.
\end{array}
\end{equation}

From the expression \eqref{eq:curv}, of the curvature tensor field of $N$,
it follows
\begin{equation}\label{eq:2}
\langle R^N(\widehat Z,Z)Z,H\rangle=\frac{\rho}{4}\{\langle Z,\widehat Z\rangle\eta(Z)\eta(H)+3\langle\widehat Z,\varphi Z\rangle\langle H,\varphi Z\rangle\}.
\end{equation}

Working just like in \cite{AdCT} (or in \cite{F}), we can prove that
\begin{equation}\label{eq:3}
\langle(\nabla^{\perp}_{Z}\sigma)(\widehat Z,Z),H\rangle=\langle\widehat Z,Z\rangle\langle\nabla^{\perp}_ZH,H\rangle.
\end{equation}

Indeed, if we consider the unit vector fields $e_1$ and $e_2$
corresponding to $\frac{\partial}{\partial u}$ and
$\frac{\partial}{\partial v}$, respectively, then we get $Z=\frac{\lambda}{\sqrt{2}}(e_1-ie_2)$ and
$$
\sigma(\widehat
Z,Z)=\frac{\lambda^2}{2}\sigma(e_1-ie_2,e_1+ie_2)=\frac{\lambda^2}{2}(\sigma(e_1,e_1)+\sigma(e_2,e_2))=\langle\widehat Z,Z\rangle H.
$$

Since we also have $\nabla_ZZ=\frac{1}{\lambda^2}\langle\nabla_ZZ,\widehat Z\rangle Z$, it follows that
$$
\begin{array}{lll}
\langle(\nabla^{\perp}_{Z}\sigma)(\widehat Z,Z),H\rangle&=&\langle\nabla^N_Z\sigma(\widehat Z,Z),H\rangle-\langle\sigma(\nabla_Z\widehat Z, Z), H\rangle-\langle\sigma(\widehat Z,\nabla_ZZ), H\rangle\\ \\
&=&\langle\nabla^N_Z(\langle\widehat Z,Z\rangle H),H\rangle-\frac{1}{\lambda^2}\langle\nabla_ZZ,\widehat Z\rangle \langle\sigma(\widehat Z,Z), H\rangle
\\ \\&=&\langle\nabla^N_Z(\langle\widehat Z,Z\rangle H),H\rangle-\langle\nabla_ZZ,\widehat Z\rangle\langle H,H\rangle\\
\\&=&\langle\nabla_Z\widehat Z,Z\rangle\langle H,H\rangle+\langle\nabla_ZZ,\widehat Z\rangle\langle H,H\rangle\\
\\&&+\langle\widehat Z,Z\rangle\langle\nabla^{\perp}_ZH,H\rangle-\langle\nabla_ZZ,\widehat Z\rangle\langle H,H\rangle\\
\\&=&\langle\widehat Z,Z\rangle\langle\nabla^{\perp}_ZH,H\rangle.
\end{array}
$$

Replacing \eqref{eq:2} and \eqref{eq:3} in \eqref{eq:1}, and using the fact that $H$ is parallel,
it follows that
\begin{equation}\label{eq:term1}
\widehat Z(\langle
\sigma(Z,Z),H\rangle)=\frac{\rho}{4}\{\langle Z,\widehat Z\rangle\eta(Z)\eta(H)+3\langle\widehat Z,\varphi Z\rangle\langle H,\varphi Z\rangle\}.
\end{equation}

As the characteristic vector field $\xi$ is parallel, equation \eqref{eq:3} also implies that
\begin{equation}\label{eq:term2}
\widehat Z((\eta(Z))^2)=2\langle Z,\widehat Z\rangle\eta(Z)\eta(H).
\end{equation}

Finally, since $\nabla^N\varphi=0$ and $H$ is parallel, using $\nabla^N_{\widehat Z}Z=\sigma(\widehat Z,Z)=\langle\widehat Z,Z\rangle H$ and $(\varphi Z)^{\top}=\frac{1}{\lambda^2}\langle
\varphi Z,\widehat Z\rangle Z$, that can be easily checked, one obtains
\begin{equation}\label{eq:term3}
\begin{array}{lll}
\widehat Z(\langle\varphi Z,H\rangle^2)&=&2\langle
\varphi Z,H\rangle\{\langle\nabla^N_{\widehat Z}\varphi Z,H\rangle+\langle\varphi Z,\nabla^N_{\widehat Z}H\rangle\}\\
\\ &=&2\langle
\varphi Z,H\rangle\{\langle\varphi \nabla^N_{\widehat Z}Z,H\rangle+\langle\varphi Z,\nabla^N_{\widehat Z}H\rangle\}\\ \\&=&2\langle\varphi Z,H\rangle\{\langle\widehat Z,Z\rangle\langle\varphi H,H\rangle-\langle(\varphi Z)^{\top},A_H\widehat Z\rangle+\langle(\varphi Z)^{\perp},\nabla^{\perp}_{\widehat Z}H\rangle\}\\ \\
&=&-2\langle\varphi Z,H\rangle\langle\sigma((\varphi Z)^{\top},\widehat Z),H\rangle\\
\\&=&-2\langle \varphi Z,H\rangle\langle\varphi Z,\widehat Z\rangle|H|^2.
\end{array}
\end{equation}

From \eqref{eq:term1}, \eqref{eq:term2} and \eqref{eq:term3} we see that $\widehat Z(Q(Z,Z))=0$, and we
can state the following.

\begin{theorem}
If $\Sigma^2$ is an immersed pmc surface in a cosymplectic space form
$N^{2n+1}(\rho)$, then the
$(2,0)$-part of the quadratic form $Q$, defined on $\Sigma^2$ by
$$
Q(X,Y)=8|H|^2\langle\sigma(X,Y),H\rangle-\rho|H|^2\eta(X)\eta(Y)+3\rho\langle \varphi X,
H\rangle\langle \varphi Y, H\rangle,
$$
is holomorphic.
\end{theorem}

\section{Vertical cylinders with parallel mean curvature vector in product spaces}\label{scyl}

Let $\gamma:I\subset\mathbb{R}\rightarrow M^n(\rho)$ be a curve parametrized by
arc-length in a complex space form with complex dimension $n$ and constant holomorphic sectional curvature $\rho$, i.e. $\mathbb{C}P^n(\rho)$, $\mathbb{C}^n$ or $\mathbb{C}H^n(\rho)$  as $\rho>0$, $\rho=0$ or $\rho<0$. The curve $\gamma$ is called a {\it Frenet curve of
osculating order} $r$, $1\leq r\leq 2n$, if there exist $r$
orthonormal vector fields $\{E_1=\gamma',\ldots,E_{r}\}$ along
$\gamma$ such that
\begin{equation}
\begin{cases}
\nabla^M_{E_{1}}E_{1}=\kappa_{1}E_{2} \\
\nabla^M_{E_{1}}E_{i}=-\kappa_{i-1}E_{i-1} + \kappa_{i}E_{i+1},
\quad \forall i=2,\dots,r-1, \\
\nabla^M_{E_{1}}E_{r}=-\kappa_{r-1}E_{r-1}
\end{cases}
\end{equation}
where $\{\kappa_{1},\kappa_{2},\kappa_{3},\ldots,\kappa_{r-1}\}$ are positive
functions on $I$ called the {\it curvatures} of $\gamma$ and
$\nabla^M$ denotes the Levi-Civita connection on $M^n(\rho)$.

A Frenet curve of osculating order $r$ is called a {\it helix of
order $r$} if $\kappa_i=\cst>0$ for $1\leq i\leq r-1$. A helix
of order $2$ is called a {\it circle}, and a helix of order $3$ is
simply called {\it helix}.

S.~Maeda and Y.~Ohnita defined in \cite{MO} the {\it
complex torsions} of the curve $\gamma$ by $\tau_{ij}=\langle
E_i, J E_j \rangle$, $1\leq i<j\leq r$, where $(J,\langle,\rangle)$ is the complex structure on $M^n(\rho)$. A helix of order $r$
is called a {\it holomorphic helix of order $r$} if all the complex
torsions are constant. It is easy to see that a circle is always a holomorphic circle.

In order to find examples of pmc surfaces we will focus our attention on the \textit{vertical cylinders} $\Sigma^2=\pi^{-1}(\gamma)$ in product spaces $M^n(\rho)\times\mathbb{R}$, where $\pi:M^n(\rho)\times\mathbb{R}\rightarrow M^n(\rho)$ is the projection map and $\gamma:I\rightarrow M^n(\rho)$ is a Frenet curve of osculating order $r$ in $M^n(\rho)$. For any vector field $X$ tangent to $M^n(\rho)$ we shall denote by $X^H$ its horizontal lift to $M^n(\rho)\times\mathbb{R}$. As for the Riemannian metrics on $M^n(\rho)$ and $M^n(\rho)\times\mathbb{R}$, we will use the same notation $\langle,\rangle$.

Obviously, $\{E_1^H,\xi\}$ is a local orthonormal frame on $\Sigma^2$ and $E_i^H$, $1<i\leq r$, are normal vector fields. Then the mean curvature vector $H$ is given by
$$
H=\frac{1}{2}(\sigma(E_1^H,E_1^H)+\sigma(\xi,\xi))=\frac{1}{2}\kappa_1E_2^H,
$$
where $\kappa_1=\kappa_1\circ\pi$ and we used the first Frenet equation for $\gamma$ and O'Neill's equation \cite{O}  in the case of cosymplectic space forms, i.e. $\nabla^N_{X^H}Y^H=(\nabla^M_XY)^H$, for any vector fields $X$ and $Y$ tangent to $M^n(\rho)$ (see also \cite{ABC}).

Next, from the second Frenet equation, we have
\begin{equation}\label{eq:PMC1}
\nabla^N_{E_1^H}H=\frac{1}{2}(\nabla^M_{E_1}(\kappa_1E_2))^H=\frac{1}{2}(\kappa_1'E_2-\kappa_1^2E_1+\kappa_1\kappa_2E_3)^H.
\end{equation}

It is easy to verify that $\nabla_{\xi}E_1^H=\nabla_{E_1^H}\xi=0$, where $\nabla$ is the connection on the surface, and then we get that $[\xi,E_1^H]=0$, which means $\nabla^N_{\xi}E_1^H=\nabla^N_{E_1^H}\xi=0$. Now, since from \eqref{eq:curv} it follows that $R^N(\xi,E_1^H)E_1^H=0$, we obtain
\begin{equation}\label{eq:PMC2}
\nabla^N_{\xi}H=\frac{1}{2}\nabla^N_{\xi}\nabla^N_{E_1^H}E_1^H=0.
\end{equation}

From \eqref{eq:PMC1} and \eqref{eq:PMC2} we see that $H$ is parallel if and only if either
\begin{itemize}
\item $\gamma$ is a geodesic in $M^n(\rho)$; or

\item $\gamma$ is a circle in $M^n(\rho)$ with the curvature $\kappa_1=2|H|=\cst>0$.
\end{itemize}
Obviously, in the first case, $\Sigma^2$ is a minimal surface. In the second case, the $(2,0)$-part of $Q$ vanishes if and only if
$$
16|H|^4+\rho|H|^2+3\rho\langle\varphi E_1^H,H\rangle^2=0,
$$
that is equivalent to
$$
4\kappa_1^2+\rho(1+3\tau_{12}^2)=0.
$$

Now, we can conclude.

\begin{proposition} A vertical cylinder $\Sigma^2=\pi^{-1}(\gamma)$ in $M^n(\rho)\times\mathbb{R}$ has non-zero parallel mean curvature vector and the $(2,0)$-part of the quadratic form $Q$ vanishes on $\Sigma^2$ if and only if $\rho<0$ and the curve $\gamma$ is a circle in $M^n(\rho)$ with the curvature $\kappa=\frac{1}{2}\sqrt{-\rho(1+3\tau^2)}$, where $\tau$ is the complex torsion of $\gamma$.
\end{proposition}

\begin{remark} S. Maeda and T. Adachi proved in \cite{MA} that for any positive number $\kappa$ and for any number $\tau$, such that $|\tau|<1$, there exits a circle with curvature $\kappa$ and complex torsion $\tau$ in any complex space form. Therefore, for any $\rho<0$, we know that circles $\gamma$, like in the previous Proposition, do exist. Since $0\leq\tau^2\leq 1$ we get that $\frac{1}{2}\sqrt{-\rho}\leq\kappa\leq\sqrt{-\rho}$, which means that the mean curvature of a non-minimal pmc cylinder $\Sigma^2=\pi^{-1}(\gamma)$, with vanishing $(2,0)$-part of $Q$, satisfies $\frac{\sqrt{-\rho}}{4}\leq|H|\leq\frac{\sqrt{-\rho}}{2}$.
\end{remark}

\section{A reduction theorem}\label{sred}

Let $\Sigma^2$ be an immersed non-minimal pmc surface in a non-flat cosymplectic space
form $N^{2n+1}(\rho)$, $n\geq 2$.

\begin{lemma}\label{lemma_com} For any vector $V$ normal to $\Sigma^2$, which is also orthogonal to
$\varphi T\Sigma^2$ and to $\varphi H$, we have $[A_H,A_V]=0$, i.e. $A_H$
commutes with $A_V$.
\end{lemma}

\begin{proof} The conclusion follows easily from the Ricci equation
$$
\langle
R^{\perp}(X,Y)H,V\rangle=\langle[A_H,A_V]X,Y\rangle+\langle
R^N(X,Y)H,V\rangle,
$$
since
$$
\begin{array}{lll}
\langle R^N(X,Y)H,V\rangle&=&\frac{\rho}{4}\{\langle X,\varphi H\rangle\langle\varphi Y,V\rangle-\langle Y,\varphi H\rangle\langle\varphi X,V\rangle+2\langle X, \varphi Y\rangle\langle\varphi H,V\rangle\}\\ \\&=&0
\end{array}
$$
and $R^{\perp}(X,Y)H=0$.
\end{proof}

\begin{corollary}\label{lemma_split} Either $H$ is an
umbilical direction or there exists a basis that diagonalizes
simultaneously $A_H$ and $A_V$, for all normal vectors $V$ satisfying $V\perp \varphi T\Sigma^2$
and $V\perp \varphi H$.
\end{corollary}

Now, assume that $H$ is an umbilical direction everywhere, which means that the surface is pseudo-umbilical, i.e. $A_H=|H|^2\id$. For such a surface, since $H$ is also parallel, we have
$$
\begin{array}{ll}
R^N(X,Y)H&=\nabla_X\nabla_YH-\nabla_Y\nabla_XH-\nabla_{[X,Y]}H\\ \\&=-|H|^2(\nabla_XY-\nabla_YX-[X,Y])=0,
\end{array}
$$
for any tangent vector fields $X$ and $Y$.

In the following we shall prove that, in this case, $\xi\perp T\Sigma^2$ and $\varphi(T\Sigma^2)\subset N\Sigma^2$, where $N\Sigma^2$ is the normal bundle of the surface.

First, we have

\begin{lemma}\label{l1umb} The following four relations are equivalent:
\begin{enumerate}
\item[({\it i})] $\xi\perp T\Sigma^2$;

\item[({\it ii})] $H\perp\xi$;

\item[({\it iii})] $\varphi(T\Sigma^2)\subset N\Sigma^2$;

\item[({\it iv})] $\varphi H\perp T\Sigma^2$.

\end{enumerate}
\end{lemma}

\begin{proof} As $H$ is umbilical, it results that $\langle\sigma(Z,Z),H\rangle=0$ and, consequently, the $(2,0)$-part of $Q$ is, in this case,
$$
Q(Z,Z)=-\rho|H|^2(\eta(Z))^2+3\rho\langle \varphi Z, H\rangle^2,
$$
where $Z$ and its conjugate $\widehat Z$ are the complex vectors on $\Sigma^2$, defined in Section \ref{sqf}.

Since $Q(Z,Z)$ is holomorphic and $H$ is umbilical and parallel, it follows that
$$
\langle Z,\widehat Z\rangle\eta(Z)\eta(H)+3\langle\varphi Z, H\rangle\langle\varphi Z,\widehat Z\rangle=0.
$$
Now, it is easy to see that $\eta(Z)\eta(H)=0$ is equivalent to $\langle\varphi Z, H\rangle\langle\varphi Z,\widehat Z\rangle=0$, and then we only have to prove the equivalence between ({\it i}) and ({\it ii}) and between ({\it iii}) and ({\it iv}), respectively.

First, if $\eta(Z)=0$ then $\eta(\nabla^N_{\widehat Z}Z)=\langle Z,\widehat Z\rangle\eta(H)=0$,
as $N^{2n+1}(\rho)$ is a cosymplectic space form and $\nabla^N_{\widehat Z}Z=\langle Z,\widehat Z\rangle H$. Conversely, if $\eta(H)=0$, we have
$$
\eta(\nabla^N_ZH)=-\eta(A_HZ)=-|H|^2\eta(Z)=0.
$$

Next, since $R^N(X,Y)H=0$, for any tangent vector fields $X$ and $Y$, we get
$$
\begin{array}{ll}
0=R^N(\widehat Z,Z)H=&\frac{\rho}{4}\{\langle\varphi H,\widehat Z\rangle\varphi Z-\langle\varphi H,Z\rangle\varphi\widehat Z+\langle\varphi Z,\widehat Z\rangle\varphi H\\ \\&+\eta(\widehat Z)\eta(H)Z-\eta(Z)\eta(H)\widehat Z\}.
\end{array}
$$
Assume that relation ({\it iii}) holds, i.e. that $\langle\varphi Z,\widehat Z\rangle=0$. As we have seen, this also implies $\eta(Z)=\eta(\widehat Z)=0$ and $\eta(H)=0$. Then, by using the definition of the cosymplectic structure on $N^{2n+1}(\rho)$, we have
$$
\langle R^N(\widehat Z,Z)H,\varphi Z\rangle=-\frac{\rho}{4}\langle Z,\widehat Z\rangle\langle\varphi H,Z\rangle=0.
$$
Conversely, if ({\it iv}) holds, i.e. if $\langle\varphi H,Z\rangle=0$, we have
$$
\begin{array}{lcl}
0&=&\langle\nabla_{\widehat Z}\varphi H,Z\rangle=\langle\varphi\nabla_{\widehat Z}H,Z\rangle+\langle\varphi H,\nabla_{\widehat Z}Z\rangle=-\langle\varphi A_H\widehat Z,Z\rangle+\langle Z,\widehat Z\rangle\langle\varphi H,H\rangle\\ \\&=&|H|^2\langle\varphi Z,\widehat Z\rangle,
\end{array}
$$
and come to the conclusion.
\end{proof}

Now, let us assume that relations ({\it i})-({\it iv}) do not hold on our surface. We choose an orthonormal basis $\{e_1,e_2\}$ on $\Sigma^2$ such that $e_1\perp\xi$, i.e. $\eta(e_1)=0$. Then, from $\langle R^N(e_1,e_2)H,e_2\rangle=0$, we obtain
$$
\langle\varphi e_2,e_1\rangle\langle\varphi H,e_2\rangle=0,
$$
which means that $\langle\varphi H,e_2\rangle=0$, and then $R^N(e_1,e_2)H=0$ can be written as
\begin{equation}\label{eq:R}
2\langle\varphi e_2,e_1\rangle\varphi H+\langle\varphi H,e_1\rangle\varphi e_2-\eta(e_2)\eta(H)e_1=0.
\end{equation}

We take the product of this equation with $\varphi H$, $e_1$ and $\varphi e_2$, respectively, and obtain
\begin{equation}\label{eq:R1}
\langle\varphi e_2,e_1\rangle\langle\varphi H,\varphi H\rangle=\eta(e_2)\eta(H)\langle\varphi H,e_1\rangle
\end{equation}
\begin{equation}\label{eq:R2}
3\langle\varphi e_2,e_1\rangle\langle\varphi H,e_1\rangle=\eta(e_2)\eta(H)
\end{equation}
and
\begin{equation}\label{eq:R3}
3\langle\varphi e_2,e_1\rangle\eta(e_2)\eta(H)=\langle\varphi H,e_1\rangle\langle\varphi e_2,\varphi e_2\rangle.
\end{equation}

Since $\langle\varphi e_2,e_1\rangle\neq 0$ and $\langle\varphi H,e_1\rangle\neq 0$, from the first two equations, we get
\begin{equation}\label{eq:R1.1}
\langle\varphi H,\varphi H\rangle=|H|^2-(\eta(H))^2=3\langle\varphi H,e_1\rangle^2
\end{equation}
and, from the last two,
\begin{equation}\label{eq:R1.2}
\langle\varphi e_2,\varphi e_2\rangle=1-(\eta(e_2))^2=9\langle\varphi e_2,e_1\rangle^2.
\end{equation}

\begin{lemma}\label{lemma_basis} If the relations ({\it i})-({\it iv}) in Lemma \ref{l1umb} do not hold on $\Sigma^2$ then we have
\begin{enumerate}
\item $2|H|^2\langle\varphi e_2,e_1\rangle=\langle\varphi H,\sigma(e_1,e_2)\rangle$;

\item $\langle\varphi H,\sigma(e_1,e_1)\rangle=\langle\varphi H,\sigma(e_2,e_2)\rangle=0$;

\item $\nabla_{e_2}e_2=\nabla_{e_2}e_1=0$;

\item $\eta(\sigma(e_1,e_2))=0$ and $\langle\varphi e_1,\sigma(e_1,e_2)\rangle=0$.
\end{enumerate}

\end{lemma}

\begin{proof} From equation \eqref{eq:R1.1} it follows
$$
2\langle\varphi H, \varphi\nabla^N_{e_2}H\rangle=6\langle\varphi H,e_1\rangle\rangle(\langle\varphi\nabla^N_{e_2}H,e_1\rangle+\langle\varphi H,\nabla^N_{e_2}e_1\rangle).
$$
But we also know that $\nabla^N_{e_2}H=-|H|^2e_2$ and, since $\langle\varphi H, e_2\rangle=0$, that $\langle\varphi H,\nabla_{e_2}e_1\rangle=0$. Replacing in the above equation and using equation \eqref{eq:R2} one obtains
$$
2|H|^2\langle\varphi e_2,e_1\rangle=\langle\varphi H,\sigma(e_1,e_2)\rangle.
$$

In the same manner, from equation \eqref{eq:R1.1}, we obtain $\langle\varphi H,\sigma(e_1,e_1)\rangle=0$ and then $\langle\varphi H,\sigma(e_2,e_2)\rangle=0$.

As $\langle\varphi H,e_2\rangle=0$ we get $\langle\varphi H,\nabla^N_{e_2}e_2\rangle=0$, which implies
$$
\langle\varphi H,\nabla_{e_2}e_2\rangle=\langle\varphi H,e_1\rangle\langle\nabla_{e_2}e_2,e_1\rangle=0,
$$
meaning that $\nabla_{e_2}e_2=0$. Since $e_1\perp e_2$, we also have $\nabla_{e_2}e_1=0$.

Finally, $\eta(e_1)=0$ and $\nabla^N\xi=0$ imply $\eta(\nabla^N_{e_2}e_1)=0$. Since $\nabla_{e_2}e_1=0$ it follows that $\eta(\sigma(e_1,e_2))=0$. Then the last identity in our Lemma follows easily by taking the product of \eqref{eq:R} with $\varphi\sigma(e_1,e_2)$.
\end{proof}

From the expression of the curvature tensor $R^N$ it can be easily checked that $R^N$ is parallel, i.e. $\nabla^NR^N=0$.  Therefore, we have $(\nabla^N_{e_1}R^N)(e_1,e_2,H)=0$ and then, as $R^N(X,Y)H=0$, for any tangent vectors $X$ and $Y$, one obtains
$$
|H|^2R^N(e_1,e_2,e_1)-R^N(\sigma(e_1,e_1),e_2,H)-R^N(e_1,\sigma(e_1,e_2),H)=0.
$$
By using \eqref{eq:curv}, \eqref{eq:R} and Lemma \ref{lemma_basis}, the above equation become, after a straightforward computation,
$$
\eta(\sigma(e_1,e_1))\eta(H)e_2-\eta(e_2)\eta(H)\sigma(e_1,e_1)+\langle\varphi H,e_1\rangle\varphi\sigma(e_1,e_2)-5|H|^2\langle\varphi e_2,e_1\rangle\varphi e_1=0,
$$
and, by taking the product with $e_2$, we obtain that
\begin{equation}\label{eq:R_par}
\eta(\sigma(e_1,e_1))\eta(H)+9|H|^2\langle\varphi e_2,e_1\rangle^2=0.
\end{equation}

Next, from equations \eqref{eq:R2}, \eqref{eq:R1.1} and \eqref{eq:R1.2}, it follows that
$$
3|H|^2\langle\varphi e_2,e_1\rangle^2=(1-6\langle\varphi e_2,e_1\rangle^2)(\eta(H))^2.
$$
Hence, replacing in \eqref{eq:R_par}, we get $\eta(\sigma(e_1,e_1))=3\eta(H)(6\langle\varphi e_2,e_1\rangle^2-1)$
and then $
\eta(\sigma(e_2,e_2))=\eta(H)(5-18\langle\varphi e_2,e_1\rangle^2)$,
which means that
\begin{equation}\label{eq:eta}
\eta(\nabla^N_{e_2}e_2)=\eta(H)(5-18\langle\varphi e_2,e_1\rangle^2),
\end{equation}
since $\nabla_{e_2}e_2=0$.

From equation \eqref{eq:R1.2}, we obtain $2\eta(e_2)\eta(\nabla^N_{e_2}e_2)=-18\langle\varphi e_2,e_1\rangle e_2(\langle\varphi e_2,e_1\rangle)$,
and then, from \eqref{eq:eta} and \eqref{eq:R2}, it results
\begin{equation}\label{eq:e2}
3e_2(\langle\varphi e_2,e_1\rangle)=(18\langle\varphi e_2,e_1\rangle^2-5)\langle\varphi H, e_1\rangle.
\end{equation}

Finally, we differentiate the equation \eqref{eq:R2}, and using the equations \eqref{eq:eta} and \eqref{eq:e2}, the fact that $H$ is umbilical and parallel and Lemma \ref{lemma_basis}, we obtain
$$
|H|^2+(5-18\langle\varphi e_2,e_1\rangle^2)(\eta(H))^2=0.
$$
But, from equation \eqref{eq:R1.2}, we know that $9\langle\varphi e_2,e_1\rangle^2<1$. Therefore, the last equation is a contradiction.

Thus, it results that $\xi\perp T\Sigma^2$,
$\varphi(T\Sigma^2)\subset N\Sigma^2$, $H\perp\xi$ and $\varphi H\perp T\Sigma^2$. Now, it is easy to see that, if $\{e_1,e_2\}$ is an orthonormal frame on $\Sigma^2$, then, at any point on the surface, the system $\{e_1,e_2,\varphi e_1,\varphi e_2, H,\varphi H,\xi\}$ is linearly independent, which means that $n\geq 3$.

Thus we can state the following

\begin{proposition}\label{p_umb} Let $\Sigma^2$ be an immersed non-minimal pmc surface in a non-flat cosymplectic space form $N^{2n+1}(\rho)$, $n\geq 2$. If the mean curvature vector $H$ is an umbilical direction everywhere, then $\xi\perp T\Sigma^2$,
$\varphi(T\Sigma^2)\subset N\Sigma^2$ and $n\geq 3$. Moreover, $H\perp\xi$ and $\varphi H\perp T\Sigma^2$.
\end{proposition}

Let $N^{2n+1}(\rho)$ be the product between a non-flat complex space form $M^n(\rho)$, with complex dimension $n$, and $\mathbb{R}$. If $\Sigma^2$ is an immersed surface in $N^{2n+1}(\rho)$ as in the previous Proposition, it follows that $\Sigma^2$ is a totally real surface in $M^n(\rho)$. Moreover, since $N^{2n+1}(\rho)$ is a product space, we have $\nabla^N_{\widehat Z}Z=\nabla^M_{\widehat Z}Z$, $\nabla^N_{Z}Z=\nabla^M_{Z}Z$ and $\nabla^N_{X}H=\nabla^M_{X}H$, for any vector field $X$ tangent to $\Sigma^2$, where we have used the fact that $H\perp\xi$. From these identities, we obtain that the surface is pseudo-umbilical and with parallel mean curvature vector in $M^n(\rho)$. Hence, we have

\begin{corollary} Let $\Sigma^2$ be an immersed non-minimal pmc surface in
$M^{n}(\rho)\times\mathbb{R}$, $n\geq 2$, $\rho\neq 0$. If its mean curvature vector is an umbilical direction everywhere, then $\Sigma^2$ is a pseudo-umbilical non-minimal totally real pmc surface in $M^n(\rho)$, and $n\geq 3$.
\end{corollary}

\begin{remark} If the mean curvature vector of the surface $\Sigma^2$ is umbilical everywhere then the $(2,0)$-part of the quadratic form $Q$ defined on $\Sigma^2$ vanishes.
\end{remark}

The next step is to study the case when the mean curvature vector of the surface is nowhere umbilical. We shall prove that such a surface lies in a totally geodesic submanifold of $N^{2n+1}(\rho)$, with dimension less than or equal to $11$.

\begin{proposition}\label{lemma_parallel} Assume that $H$ is nowhere an umbilical direction.
Then there exists a parallel subbundle of the normal bundle that
contains the image of the second fundamental form $\sigma$ and has
dimension less than or equal to $9$.
\end{proposition}

\begin{proof} We consider a subbundle $L$ of the normal bundle, given by
$$
L=\Span\{\im\sigma\cup(\varphi(\im\sigma))^{\perp}\cup(\varphi(T\Sigma^2))^{\perp}\cup\xi^{\perp}\},
$$
where
$(\varphi(T\Sigma^2))^{\perp}=\{(\varphi X)^{\perp}:X\ \textnormal{tangent to}\
\Sigma^2\}$, $(\varphi(\im\sigma))^{\perp}=\{(\varphi\sigma(X,Y))^{\perp}:X,Y\ \textnormal{tangent to}\
\Sigma^2\}$ and $\xi^{\perp}$ is the normal component of $\xi$ along the surface. We will show that $L$ is parallel.

First, we have to prove that if $V$ is orthogonal to $L$, then
$\nabla^{\perp}_{e_i}V$ is orthogonal to $\varphi(T\Sigma^2)$ and to $\varphi H$,
where $\{e_1,e_2\}$ is a frame satisfying
$$
\langle\sigma(e_1,e_2),V\rangle=\langle\sigma(e_1,e_2),H\rangle=0.
$$
Indeed, one gets
$$
\begin{array}{lll}
\langle(\varphi H)^{\perp},\nabla^{\perp}_{e_i}V\rangle
&=&\langle(\varphi H)^{\perp},\nabla^N_{e_i}V\rangle
=-\langle\nabla^N_{e_i}(\varphi H)^{\perp},V\rangle\\
\\&=&-\langle\nabla^N_{e_i}\varphi H,V\rangle+\langle\nabla^N_{e_i}(\varphi H)^{\top},V\rangle\\
\\&=&\langle\varphi A_He_i,V\rangle+\langle\sigma(e_i,(\varphi H)^{\top}),V\rangle\\ \\&=&0
\end{array}
$$
and
$$
\begin{array}{lll}
\langle(\varphi e_j)^{\perp},\nabla^{\perp}_{e_i}V\rangle&=&-\langle\nabla^N_{e_i}(\varphi e_j)^{\perp},V\rangle\\
\\&=&-\langle\nabla^N_{e_i}\varphi e_j,V\rangle+\langle\nabla^N_{e_i}(\varphi e_j)^{\top},V\rangle\\
\\&=&-\langle\varphi\nabla_{e_i}e_j,V\rangle-\langle\varphi\sigma(e_i,e_j),V\rangle+\langle\sigma(e_i,(\varphi e_j)^{\top}),V\rangle\\
\\&=&0.
\end{array}
$$

Next, we shall prove that if a normal vector $V$ is orthogonal
to $L$, then so is $\nabla^{\perp}V$, i.e.
$$
\langle\sigma(e_i,e_j),\nabla^{\perp}_{e_k}V\rangle=0,\quad
\langle
\varphi\sigma(e_i,e_j),\nabla^{\perp}_{e_k}V\rangle=0,
$$
$$
\langle\varphi e_i,\nabla^{\perp}_{e_k}V\rangle=0,\quad\langle\xi^{\perp},\nabla^{\perp}_{e_k}V\rangle=0.
$$
We only have to prove the first two identities and the last one, since the third has been obtained above.

Let us denote
$A_{ijk}=\langle\nabla^{\perp}_{e_k}\sigma(e_i,e_j),V\rangle$.
As $\sigma$ is symmetric, we have $A_{ijk}=A_{jik}$, and also
$A_{ijk}=-\langle\sigma(e_i,e_j),\nabla^{\perp}_{e_k}V\rangle$,
since $V$ is orthogonal to $L$. We get
$$
\begin{array}{lll}
\langle(\nabla^{\perp}_{e_k}\sigma)(e_i,e_j),V\rangle&=&
\langle\nabla^{\perp}_{e_k}\sigma(e_i,e_j),V\rangle-\langle\sigma(\nabla_{e_k}e_i,e_j),V\rangle
-\langle\sigma(e_i,\nabla_{e_k}e_j),V\rangle\\
\\&=&\langle\nabla^{\perp}_{e_k}\sigma(e_i,e_j),V\rangle,
\end{array}
$$
and, from the Codazzi equation, again using $V\perp L$,
$$
\begin{array}{lll}
\langle(\nabla^{\perp}_{e_k}\sigma)(e_i,e_j),V\rangle&=&\langle(\nabla^{\perp}_{e_i}\sigma)(e_k,e_j)+(R^N(e_k,e_i)e_j)^{\perp},V\rangle\\
\\&=&\langle(\nabla^{\perp}_{e_j}\sigma)(e_k,e_i)+(R^N(e_k,e_j)e_i)^{\perp},V\rangle\\
\\&=&\langle(\nabla^{\perp}_{e_i}\sigma)(e_k,e_j),V\rangle=\langle(\nabla^{\perp}_{e_j}\sigma)(e_k,e_i),V\rangle.
\end{array}
$$
We have just proved that
$A_{ijk}=A_{kji}=A_{ikj}$.

Next, since $\nabla^{\perp}_{e_k}V$ is orthogonal to $\varphi(T\Sigma^2)$
and to $\varphi H$, it follows that the frame field $\{e_1,e_2\}$
diagonalizes $A_{\nabla^{\perp}_{e_k}V}$ as well, and we get
$$
A_{ijk}=-\langle\sigma(e_i,e_j),\nabla^{\perp}_{e_k}V\rangle=-\langle
e_i,A_{\nabla^{\perp}_{e_k}V}e_j\rangle=0
$$
for any $i\neq j$. Hence, we have obtained that if two
indices are different from each other then $A_{ijk}=0$.

Next, we have
$$
\begin{array}{lll}
A_{iii}&=&-\langle\sigma(e_i,e_i),\nabla^{\perp}_{e_i}V\rangle
=-\langle 2H,\nabla^{\perp}_{e_i}V\rangle+\langle\sigma(e_j,e_j),\nabla^{\perp}_{e_i}V\rangle\\
\\&=&\langle 2\nabla^{\perp}_{e_i}H,V\rangle-A_{jji}=0,
\end{array}
$$
and, therefore, the first identity is proved.

In order to obtain the second one, we observe first that if $V$ is orthogonal to $L$ then also $\varphi V$ is normal and orthogonal to $L$. It follows that
$$
\begin{array}{lll}
\langle(\varphi\sigma(e_i,e_j))^{\perp},\nabla^{\perp}_{e_k}
V\rangle &=&-\langle\nabla^N_{e_k}(\varphi\sigma(e_i,e_j))^{\perp},V\rangle\\
\\&=&-\langle\nabla^N_{e_k}\varphi\sigma(e_i,e_j),V\rangle+
\langle\nabla^N_{e_k}(\varphi\sigma(e_i,e_j))^{\top},V\rangle\\ \\
&=&\langle\varphi A_{\sigma(e_i,e_j)}e_k,V\rangle-\langle\varphi\nabla^{\perp}_{e_k}\sigma(e_i,e_j),V\rangle\\
\\&&+\langle\sigma(e_k,(\varphi\sigma(e_i,e_j))^{\top}),V\rangle\\ \\&=&\langle\nabla^{\perp}_{e_k}\sigma(e_i,e_j),\varphi V\rangle=-\langle\sigma(e_i,e_j),\nabla^{\perp}_{e_k}\varphi V\rangle
\\ \\&=&0.
\end{array}
$$

Finally, we get
$$
\begin{array}{ll}
\langle\xi^{\perp},\nabla^{\perp}_{e_k}V\rangle&=\langle\xi^{\perp},\nabla^N_{e_k}V\rangle
=-\langle\nabla^N_{e_k}\xi^{\perp},V\rangle\\ \\&=-\langle\nabla^N_{e_k}\xi,V\rangle+\langle\nabla^N_{e_k}\xi^{\top},V\rangle=\langle \sigma(e_k,\xi^{\top}),V\rangle\\ \\&=0,
\end{array}
$$
which completes the proof.
\end{proof}

Since $\varphi(L\oplus T\Sigma^2)\subset L\oplus T\Sigma^2$ and $\xi\in L\oplus T\Sigma^2$ along the surface, it follows that $R^N(X,Y)Z\in L\oplus T\Sigma^2$ for any $X,Y,Z\in L\oplus T\Sigma^2$. Therefore, by using a result of J. H. Eschenburg and R. Tribuzy (Theorem 2 in \cite{ET}) and the result of H. Endo in \cite{E}, we get

\begin{proposition}\label{p_numb} Let $\Sigma^2$ be an immersed non-minimal pmc surface in a non-flat cosymplectic space form
$N^{2n+1}(\rho)$, $n\geq 2$. If its mean curvature vector is nowhere an umbilical direction, then the surface lies in a cosymplectic space form $N^{r}(\rho)$, where $r\leq 11$.
\end{proposition}

If we consider the cosymplectic space form $N^{2n+1}(\rho)$ to be the product between a complex space form $M^{n}(\rho)$ and $\mathbb{R}$ and use again the facts that $\varphi(L\oplus T\Sigma^2)\subset L\oplus T\Sigma^2$ and $\xi\in L\oplus T\Sigma^2$, then we have the following

\begin{corollary} Let $\Sigma^2$ be an immersed non-minimal pmc surface in
$M^{n}(\rho)\times\mathbb{R}$, $n\geq 2$, $\rho\neq 0$. If its mean curvature vector is nowhere an umbilical direction, then the surface lies in $M^{r}(\rho)\times\mathbb{R}$, where $r\leq 5$.
\end{corollary}

\begin{remark}\label{rem_split} Since the map $p\in\Sigma^2\rightarrow(A_H-\mu\id)(p)$,
where $\mu$ is a constant, is analytic, it follows that if $H$ is
an umbilical direction, then this either holds on $\Sigma^2$ or
only for a closed set without interior points. In this second case
$H$ is not an umbilical direction in an open dense set, and then Proposition \ref{lemma_parallel} holds on this set. By continuity it holds on $\Sigma^2$. Consequently, only the two above studied cases can occur.
\end{remark}

Summarizing, we can state

\begin{theorem}\label{th_red} Let $\Sigma^2$ be an immersed non-minimal pmc surface in a non-flat cosymplectic
space form $N^{2n+1}(\rho)$, $n\geq 2$. Then, one of the following holds:
\begin{enumerate}

\item $\Sigma^2$ is pseudo-umbilical and then $\xi\perp T\Sigma^2$,
$\varphi(T\Sigma^2)\subset N\Sigma^2$, $H\perp\xi$, $\varphi H\perp T\Sigma^2$ and $n\geq 3$; or

\item $\Sigma^2$ is not pseudo-umbilical and lies in a cosymplectic space form $N^{r}(\rho)$, where $r\leq 11$.
\end{enumerate}
\end{theorem}

\begin{corollary}\label{c_red} Let $\Sigma^2$ be an immersed non-minimal pmc surface in
$N^{2n+1}(\rho)=M^{n}(\rho)\times\mathbb{R}$, where $M^n(\rho)$ is a non-flat complex space form, with complex dimension $n\geq 2$. Then one of the following holds:
\begin{enumerate}
\item $\Sigma^2$ is pseudo-umbilical in $N^{2n+1}(\rho)$ and then it is a pseudo-umbilical non-minimal totally real pmc surface in $M^n(\rho)$ and $n\geq 3$; or

\item $\Sigma^2$ is not pseudo-umbilical in $N^{2n+1}(\rho)$ and then it lies in $M^{r}(\rho)\times\mathbb{R}$, where $r\leq 5$.

\end{enumerate}

\end{corollary}

\section{Anti-invariant pmc surfaces}\label{santi}

Let $\Sigma^2$ be an immersed non-minimal anti-invariant pmc surface in a non-flat cosymplectic space form $N^{2n+1}(\rho)$ and define a new quadratic form $Q'$ on $\Sigma^2$ by
$$
Q'(X,Y)=8\langle\sigma(X,Y),H\rangle-\rho\eta(X)\eta(Y).
$$
In the same way as in Section \ref{sqf} it can be proved that the $(2,0)$-part of $Q'$ is holomorphic.

In the following, we shall assume that the $(2,0)$-parts of $Q$ and $Q'$ vanish on the surface, i.e. the following equations hold on $\Sigma^2$:
\begin{equation}\label{eq:q'e}
\begin{cases}
8|H|^2\langle\sigma(e_1,e_1)-\sigma(e_2,e_2),H\rangle-\rho|H|^2((\eta(e_1))^2-(\eta(e_2))^2)\\+3\rho(\langle\varphi e_1,H\rangle^2-\langle\varphi e_2,H\rangle^2)=0\\8|H|^2\langle\sigma(e_1,e_2),H\rangle-\rho|H|^2\eta(e_1)\eta(e_2)+3\rho\langle\varphi e_1,H\rangle\langle\varphi e_2,H\rangle)=0
\end{cases}
\end{equation}
and
\begin{equation}\label{eq:q''e}
\begin{cases}
8\langle\sigma(e_1,e_1)-\sigma(e_2,e_2),H\rangle-\rho((\eta(e_1))^2-(\eta(e_2))^2)=0\\
8\langle\sigma(e_1,e_2),H\rangle-\rho\eta(e_1)\eta(e_2)=0,
\end{cases}
\end{equation}
where $\{e_1,e_2\}$ is an orthonormal frame on the surface.

From \eqref{eq:q''e} it results that $\xi$ is orthogonal to the surface at a point $p$ if and only if $H$ is an umbilical direction at $p$.  Therefore, using Remark \ref{rem_split}, we obtain that either $\xi$ is orthogonal to the surface at any point or this holds only in a closed set without interior points. From Theorem \ref{th_red} we know that the first case is possible only for $n\geq 3$.

Next, if $\xi_p$ is tangent to the surface at any point $p$ in an open, connected subset of $\Sigma^2$, it follows that the Gaussian curvature $K$ of $\Sigma^2$ vanishes on this set, since $\xi$ is parallel. Therefore, $K$ vanishes on the whole surface, and this cannot occur for $2$-spheres. We however studied this case in Section \ref{scyl}, where $N^{2n+1}(\rho)$ is the product between a non-flat complex space form and the Euclidean line $\mathbb{R}$. In general, for a surface in an arbitrary cosymplectic space form $N^{2n+1}(\rho)$, we can choose an orthonormal frame $\{e_1,\xi\}$ on the surface, and easily see that $\sigma(\xi,\xi)=0$, $\sigma(e_1,\xi)=0$ and $\sigma(e_1,e_1)=2H$. Moreover, from \eqref{eq:q'e} and \eqref{eq:q''e}, we have that $H\perp\varphi e_1$ and $\rho=-16|H|^2$.

\begin{remark} We shall use now an argument in \cite{AdCT}, in order to show that either $\xi$ is tangent to $\Sigma^2$ everywhere or this holds only in a closed set without interior points. Let $f:\Sigma^2\rightarrow\mathcal{L}(N\Sigma^2,\mathbb{R})$ be the map that takes any point $p\in\Sigma^2$ to the linear function $f_p$ on $N_p\Sigma^2$, given by $f_p(X_p)=\eta_p(X_p)$, for any normal vector $X_p$ at $p$. Obviously, $\xi$ is tangent to the surface at $p\in\Sigma^2$ if and only if $f_p$ vanishes identically. By analyticity, either $f$ is identically zero on the surface or the set of its zeroes is closed and without interior points. 
\end{remark}

Now, in order to treat the case where $\xi$ has non-vanishing tangent and normal components in an open dense set $T\subset\Sigma^2$, we shall split our study in two cases, as $n=2$ or $n\geq 3$. We will work in the open dense set $T$ and all results obtained below, that hold on this set, actually hold on $\Sigma^2$, by continuity. 
 
{\bf\underline{Case I: $n=2$.}} Let us consider the orthonormal basis $\{e_1,e_2\}$ in $T_p\Sigma^2$ for any $p\in T$, where $e_2=\frac{\xi^{\top}}{|\xi^{\top}|}$ is the unit vector in the direction of the projection of $\xi$ on the tangent space. Then, since $\eta(e_1)=0$ and, from \eqref{eq:q'e} and \eqref{eq:q''e}, we have $\varphi e_1\perp H$ and $\varphi e_2\perp H$, it follows that $\{e_1,e_2,e_3=\varphi e_1,e_4=\frac{\varphi e_2}{|\varphi e_2|},e_5=\frac{H}{|H|}\}$ is an orthonormal basis in $T_pN^5$. Observe that, at any point $p\in T$, the characteristic vector field $\xi$ can be written as
\begin{equation}\label{eq:xi}
\xi=\mu e_2+\nu e_5,
\end{equation}
where $\mu=\eta(e_2)$ and $\nu=\eta(e_5)=\frac{\eta(H)}{|H|}$, is called the \textit{angle function}.

Next, from the second equation of \eqref{eq:q''e}, we get that $\{e_1,e_2\}$ diagonalizes $A_H$. Moreover, using the Ricci equation, one obtains that $\{e_1,e_2\}$ also diagonalizes $A_{\varphi e_1}$ and $A_{\varphi e_2}$, since $\langle R^N(e_1,e_2)H,\varphi e_1\rangle=\langle R^N(e_1,e_2)H,\varphi e_2\rangle=0$.

Finally, the first equation of \eqref{eq:q''e} leads to
\begin{equation}\label{eq:ah}
A_{e_5}=\left(\begin{array}{cc}\lambda_1&0\\0&\lambda_2\end{array}\right)
=\left(\begin{array}{cc}|H|(1-\frac{\rho}{16|H|^2}\mu^2)&0
\\0&|H|(1+\frac{\rho}{16|H|^2}\mu^2)\end{array}\right).
\end{equation}

\begin{lemma}\label{lemma:calc5} The following identities hold:
\begin{enumerate}
\item $e_1(\mu)=e_1(\nu)=0$;

\item $e_2(\mu)=\lambda_2\nu$ and $e_2(\nu)=-\lambda_2\mu$;

\item $\nabla_{e_1}e_1=-\lambda_1\frac{\nu}{\mu}e_2$ and $\nabla_{e_2}e_2=0$;

\item $\sigma(e_i,e_i)=\frac{\lambda_i}{|H|}H$, $i\in\{1,2\}$.

\end{enumerate}
\end{lemma}

\begin{proof} The fact that $\xi$ is parallel, \eqref{eq:xi} and \eqref{eq:ah} imply that
$$
\begin{array}{lll}
0&=&\nabla^N_{e_1}\xi=\nabla^N_{e_1}(\mu e_2+\nu e_5)
\\ \\&=&e_1(\mu)e_2+\mu\nabla_{e_1}e_2-\nu A_{e_5}e_1+\mu\sigma(e_1,e_2)+e_1(\nu)e_5
\\ \\&=&e_1(\mu)e_2+\mu\nabla_{e_1}e_2-\lambda_1\nu e_1+e_1(\nu)e_5.
\end{array}
$$
The tangent and the normal part in the right hand side vanish and then, since $\nabla_{e_1}e_2\perp e_2$, it follows that $e_1(\mu)=e_1(\nu)=0$ and $\nabla_{e_1}e_2=\lambda_1\frac{\nu}{\mu}e_1$. As $\langle\nabla_{e_1}e_2,e_1\rangle+\langle\nabla_{e_1}e_1,e_2\rangle=0$ and $\nabla_{e_1}e_1\perp e_1$, the last identity is equivalent to $\nabla_{e_1}e_1=-\lambda_1\frac{\nu}{\mu}e_2$.

In the same way, we get
$$
\begin{array}{lll}
0&=&\nabla^N_{e_2}\xi=\nabla^N_{e_2}(\mu e_2+\nu e_5)
\\ \\&=&e_2(\mu)e_2+\mu\nabla_{e_2}e_2-\nu A_{e_5}e_2+\mu\sigma(e_2,e_2)+e_2(\nu)e_5
\\ \\&=&e_2(\mu)e_2+\mu\nabla_{e_2}e_2-\lambda_2\nu e_2+e_2(\nu)e_5+\mu\sigma(e_2,e_2)
\end{array}
$$
and then $\nabla_{e_2}e_2=0$, $e_2(\mu)=\lambda_2\nu$ and, since $\mu^2+\nu^2=1$, $e_2(\nu)=-\lambda_2\mu$. We also obtain that $\sigma(e_2,e_2)=-\frac{e_2(\nu)}{\mu}e_5=\frac{\lambda_2}{|H|}H$ and $\sigma(e_1,e_1)=2H-\sigma(e_2,e_2)=\frac{\lambda_1}{|H|}H$.
\end{proof}

\begin{remark} A direct consequence of the previous Lemma is that $A_{\varphi e_1}$ and $A_{\varphi e_2}$ vanish and then the only non-zero component of $A$ is $A_{e_5}$.
\end{remark}

Now, assume that the characteristic vector field $\xi$ is either tangent to the surface or it has non vanishing tangent and normal components in an open dense set $T\subset\Sigma^2$, and consider the subbundle of the normal bundle $L=\im\sigma$. It is easy to see that $L$ is parallel, $\dim L=1$, $\varphi X\perp Y$, for any $X,Y\in T\Sigma^2\oplus L$ and that $T\Sigma^2\oplus L$ is invariant by $R^N$, since $\xi\in T\Sigma^2\oplus L$ along the surface.

On the other hand, any non-minimal cmc surface immersed in an anti-invariant totally geodesic $3$-dimensional submanifold of $N^5(\rho)$ is an immersed non-minimal anti-invariant pmc surface in $N^5(\rho)$. Moreover, if we assume that the $(2,0)$-part of $Q'$ vanishes on such a surface, it follows that also the $(2,0)$-part of $Q$ vanishes.

Therefore, using Theorem $2$ in \cite{ET}, we get

\begin{theorem} A surface $\Sigma^2$ can be immersed as a non-minimal anti-invariant pmc surface in a non-flat cosymplectic space form $N^5(\rho)$, with vanishing $(2,0)$-parts of the quadratic forms $Q$ and $Q'$, if and only if $\Sigma^2$ is an immersed non-minimal cmc surface in a $3$-dimensional totally geodesic anti-invariant submanifold of $N^5(\rho)$, such that the $(2,0)$-part of $Q'$ vanishes.
\end{theorem}

The $3$-dimensional totally geodesic anti-invariant submanifolds of $M^2(\rho)\times\mathbb{R}$, where $M^2$ is a non-flat complex space form, are $\bar M^2\times\mathbb{R}$, where $\bar M^2$ is a totally geodesic Lagrangian submanifold of $M^2(\rho)$. B.-Y. Chen and K. Ogiue proved in \cite{CO} (Proposition 3.2) that a totally geodesic totally real submanifold $\bar M^m$ of a non-flat complex space form $M^n(\rho)$ is necessarily a space form with constant curvature $\frac{\rho}{4}$. Moreover, it is known that $\mathbb{S}^2(\frac{\rho}{4})$ and $\mathbb{H}^2(\frac{\rho}{4})$ can be isometrically immersed as totally geodesic Lagrangian submanifolds in $\mathbb{C}P^2(\rho)$ and $\mathbb{C}H^2(\rho)$, respectively (see \cite{CMU}).

Hence, an immersed non-minimal anti-invariant surface pmc surface in $M^2(\rho)\times\mathbb{R}$ on which the $(2,0)$-parts of $Q$ and $Q'$ vanish, is a non-minimal cmc surface in $\bar M^2(\frac{\rho}{4})\times\mathbb{R}$ with vanishing $(2,0)$-part of $Q'$, which in this case is just the Abresch-Rosenberg differential introduced in \cite{AR}, where $\bar M^2(\frac{\rho}{4})$ is a complete simply-connected surface with constant curvature $\frac{\rho}{4}$. U. Abresch and H. Rosenberg proved there are four classes of such surfaces, the first three of them, namely the cmc spheres $S_H^2\subset\bar M^2(\frac{\rho}{4})\times\mathbb{R}$ of Hsiang and Pedrosa, their non-compact cousins $D_H^2$ and the surfaces of catenoidal type $C_H^2$, being embedded and rotationally invariant, and the fourth one being comprised of parabolic surfaces $P_H^2$ (see \cite{AR} and \cite{AR2} for detailed description of all these surfaces).

\begin{corollary}\label{cor:class} Any immersed non-minimal anti-invariant pmc surface in $M^2(\rho)\times\mathbb{R}$ with vanishing $(2,0)$-parts of the quadratic forms $Q$ and $Q'$ is one of the surfaces $S_H^2$, $D_H^2$, $C_H^2$ and $P_H^2$ in the product space $\bar M^2(\frac{\rho}{4})\times\mathbb{R}$.
\end{corollary}

Therefore, we have

\begin{theorem} Any immersed non-minimal anti-invariant pmc $2$-sphere in a non-flat cosymplectic space form $M^2(\rho)\times\mathbb{R}$ is one of the embedded rotationally invariant cmc spheres $S_H^2\subset\bar M^2(\frac{\rho}{4})\times\mathbb{R}$.
\end{theorem}

\begin{remark} A surface $\Sigma^2$ immersed in a cosymplectic space form is called a \textit{slant surface} if for all vectors $X$ tangent to $\Sigma^2$ and orthogonal to $\xi$ the angle $\theta$ between $\varphi X$ and $T_p\Sigma^2$ is constant, i.e. $\theta$ does not depend on $X$ or on the point $p$ on the surface. Obviously, the invariant and anti-invariant surfaces are slant surfaces. A slant surface which is neither invariant nor anti-invariant is called a proper slant surface. If $\Sigma^2$ is a proper slant surface then $\xi$ is orthogonal to the surface (see \cite{L}). It follows that, if $\Sigma^2$ is an immersed proper slant surface in $M^2(\rho)\times\mathbb{R}$, then it lies in $M^2(\rho)$. On the other hand, there are no non-minimal pmc $2$-spheres in a non-flat complex space form $M^2(\rho)$ (see \cite{H}). Therefore, $S_H^2\subset\bar M^2(\frac{\rho}{4})\times\mathbb{R}$ are the only non-minimal slant pmc $2$-spheres in $M^2(\rho)\times\mathbb{R}$.
\end{remark}

In the following, we shall see that Lemma \ref{lemma:calc5} allows us to make some considerations about the admissible range of the angle function $\nu$.

Let $\Sigma^2$ be a surface as in Corollary \ref{cor:class} with parallel mean curvature vector $H$. From Lemma \ref{lemma:calc5}, it follows, after a straightforward computation, that
\begin{equation}\label{D1}
\Delta\nu^2=2\lambda_2^2(1-3\nu^2)
\end{equation}
and
\begin{equation}\label{D2}
\Delta|A|^2=\frac{\rho^2}{32|H|^2}\lambda_2^2\mu^2(5\nu^2-1).
\end{equation}

Assume now that the surface is complete and $K\geq 0$, so that $\Sigma^2$ is a parabolic space.

If $\nu^2\geq\frac{1}{5}$ on an open dense subset of $\Sigma^2$, then, from \eqref{D2}, it follows that $|A|^2$ is a subharmonic function, and, since $|A|^2$ is bounded by \eqref{eq:ah}, we get that either $\lambda_2^2=0$ or $\mu^2=0$ or $\nu^2=\frac{1}{5}$. J. M. Espinar and H. Rosenberg proved in \cite{ER} that if the angle function $\nu$ is constant, then $\nu^2=0$ or $\nu^2=1$, the second case being possible only when the surface is minimal. Therefore, since we also know that $\mu^2$ cannot vanish on an open dense subset of $\Sigma^2$, one obtains that $\lambda_2^2=|H|^2(1+\frac{\rho}{16|H|^2}\mu^2)^2=0$, and then that $\mu$ and $\nu$ are constant, which means that $\nu^2=0$ and $\mu^2=1$. But this is a contradiction, since we assumed that $\nu^2\geq\frac{1}{5}$.

If $\nu^2\leq\frac{1}{3}$ on an open dense subset of $\Sigma^2$, then, from \eqref{D1}, in the same way as above, we obtain that $\nu^2=0$, $K=0$ and $\rho=-16|H|^2$. In this case, $\Sigma^2$ is a vertical cylinder over a circle in $\mathbb{H}^2(-4|H|^2)$, with curvature $\kappa=2|H|$ and complex torsion equal to $0$ (see also \cite{ER}).

Next, if $\Sigma^2$ is compact, from \eqref{D2} and the divergence theorem, we get that if $\nu^2\geq\frac{1}{5}$ then $\nu=0$, which is a contradiction. From \eqref{D1}, again using the divergence theorem, we obtain that, if  $\nu^2\leq\frac{1}{3}$ on $\Sigma^2$, then the surface is a cylinder, which is also a contradiction, since we assumed that $\Sigma^2$ is compact.

Summarizing, we proved the following

\begin{proposition} Let $\Sigma^2$ be a complete non-minimal cmc surface in $\bar M^2(\frac{\rho}{4})\times\mathbb{R}$ with vanishing Abresch-Rosenberg differential and non-negative Gaussian curvature. Then we have that:
\begin{enumerate}
\item $\nu^2\geq\frac{1}{5}$ cannot occur on an open dense subset of $\Sigma^2$;

\item if $\nu^2\leq\frac{1}{3}$ on an open dense subset of $\Sigma^2$, then $\nu$ vanishes identically and the surface is a vertical cylinder over a circle in $\mathbb{H}^2(-4|H|^2)$, with curvature $\kappa=2|H|$ and complex torsion equal to $0$.
\end{enumerate}
\end{proposition}

\begin{proposition} There are no compact non-minimal cmc surfaces in $\bar M^2(\frac{\rho}{4})\times\mathbb{R}$ with vanishing Abresch-Rosenberg differential, such that one of the inequalities $\nu^2\geq\frac{1}{5}$ or $\nu^2\leq\frac{1}{3}$ holds on the surface.
\end{proposition}

{\bf\underline{Case II: $n\geq 3$.}} We note first that, according to Theorem \ref{th_red}, the surface cannot be pseudo-umbilical, since we have assumed that the tangent part of $\xi$ does not vanish in an open dense set.

Now, let us consider again the orthonormal basis $\{e_1,e_2\}$ in $T_p\Sigma^2$, $p\in T$, where $e_2$ is the unit vector in the direction of the projection of $\xi$ on the tangent space. From \eqref{eq:q'e} and \eqref{eq:q''e}, we can see that $\{e_1,e_2\}$ diagonalizes $A_H$ in this case too. Since the surface is anti-invariant, from the Ricci equation, we get $[A_H,A_V]=0$, for any normal vector $V$ and, therefore, $\{e_1,e_2\}$ diagonalizes $A_V$, for any normal vector $V$. We define the subbundle $L=\Span\{\im\sigma\cup\xi^{\perp}\}$ in the normal bundle and, in the same way as in Lemma \ref{lemma_parallel}, we can prove that, for any normal vector $V$, orthogonal to $L$, we have $\langle\sigma(e_i,e_j),\nabla^{\perp}_{e_k}V\rangle=0$ and $\langle\xi^{\perp},\nabla^{\perp}_{e_k}V\rangle=0$, $i,j,k=\{1,2\}$, which means that $L$ is parallel. It is also easy to see that $T\Sigma^2\oplus L$ is invariant by $R^N$. We shall prove that $\varphi X\perp Y$, for any $X,Y\in T\Sigma^2\oplus L$. Since the surface is anti-invariant, we have $\varphi e_1\perp e_2$ and, moreover, $\langle\varphi e_1,\xi^{\perp}\rangle=\langle\varphi e_1,\xi-\xi^{\top}\rangle=0$. Next, we obtain
$$
\langle\varphi e_1,\sigma(e_2,e_2)\rangle=\langle\varphi e_1,\nabla^N_{e_2}e_2\rangle=-\langle\varphi\nabla^N_{e_2}e_1,e_2\rangle=
-\langle\varphi\nabla_{e_2}e_1,e_2\rangle=0,
$$
again using the fact that $\Sigma^2$ is anti-invariant and $\sigma(e_1,e_2)=0$. From the equations \eqref{eq:q'e} and \eqref{eq:q''e} it follows that $\varphi e_i\perp H$, $i=\{1,2\}$, and then
$$
\langle\varphi e_1,\sigma(e_1,e_1)\rangle=\langle\varphi e_1,2H-\sigma(e_2,e_2)\rangle=0.
$$
Since $T\Sigma^2\oplus L=\Span\{e_1,e_2,\sigma(e_1,e_1),\sigma(e_2,e_2),\xi^{\perp}\}$, we have just proved that $\varphi e_1$ is orthogonal to $T\Sigma^2\oplus L$. In the same way we get that $\varphi e_2$ and $\varphi\xi^{\perp}=|\xi^{top}|\varphi e_2$ are orthogonal to $T\Sigma^2\oplus L$. Finally, since $\varphi H\perp e_i$, $i=\{1,2\}$, it results that $\varphi H$ is normal and one gets
$$
\begin{array}{ll}
\langle\varphi\sigma(e_1,e_1),\sigma(e_2,e_2)\rangle&=\langle\varphi\sigma(e_1,e_1),2H-\sigma(e_1,e_1)\rangle
=2\langle\varphi\sigma(e_1,e_1),H\rangle\\ \\&=-2\langle\nabla_{e_1}^N e_1,\varphi H\rangle=2\langle e_1,\varphi\nabla_{e_1}^NH\rangle=-2\langle e_1,\varphi A_He_1\rangle\\ \\&=0,
\end{array}
$$
which means $\varphi\sigma(e_i,e_i)\perp T\Sigma^2\oplus L$.

Hence, $T\Sigma^2\oplus L$ is parallel, invariant by $R^N$, anti-invariant by $\varphi$ and its dimension is less than or equal to $5$. Now, again using Theorem $2$ in \cite{ET}, we can state

\begin{theorem} A non-minimal non-pseudo-umbilical anti-invariant pmc surface immersed in a non-flat cosymplectic space form $N^{2n+1}(\rho)$, $n\geq 3$, with vanishing $(2,0)$-parts of $Q$ and $Q'$, lies in a totally geodesic anti-invariant submanifold of $N^{2n+1}(\rho)$, with dimension less than or equal to $5$.
\end{theorem}

If $N^{2n+1}(\rho)$ is of product type, we use again Proposition 3.2 in \cite{CO}, in order to obtain

\begin{corollary} A non-minimal non-pseudo-umbilical anti-invariant pmc surface immersed in $M^n(\rho)\times\mathbb{R}$, $n\geq 3$, $\rho\neq 0$, with vanishing $(2,0)$-parts of $Q$ and $Q'$, lies in a product space $\bar M^4(\frac{\rho}{4})\times\mathbb{R}$, where $\bar M^4(\frac{\rho}{4})$ is a space form immersed as a totally geodesic totally real submanifold in the complex space form $M^n(\rho)$.
\end{corollary}

\begin{remark} The non-minimal non-pseudo-umbilical pmc $2$-spheres immersed in $\bar M^4(\frac{\rho}{4})\times\mathbb{R}$ were characterized by H.~Alencar, M.~do Carmo and R.~Tribuzy  in \cite{AdCT} (Theorem 2(4)). In the same paper, they also described the non-minimal non-pseudo-umbilical complete pmc surfaces with non-negative Gaussian curvature with vanishing $(2,0)$-part of $Q'$ (Theorem 3(4)).
\end{remark}

\begin{remark} As we have seen, a proper slant surface $\Sigma^2$ immersed in $M^n(\rho)\times\mathbb{R}$, $\rho\neq 0$, lies in $M^n(\rho)$. Moreover, as an immersed surface in this space, it has constant K\"ahler angle. In \cite{F} it is proved that there are no non-minimal non-pseudo-umbilical pmc $2$-spheres with constant K\"ahler angle in a non-flat complex space form. Therefore, there are no non-minimal non-pseudo-umbilical proper slant pmc $2$-spheres in $M^n(\rho)\times\mathbb{R}$.
\end{remark}

\end{document}